\documentclass{amsart}
\usepackage{amsmath,amssymb,amsthm,amscd}
\usepackage[a4paper,text={140mm,240mm},centering,headsep=5mm,footskip=10mm]{geometry}
\usepackage[matrix,arrow, curve]{xy}
\numberwithin{equation}{section}
\usepackage{color}
\usepackage{graphicx}
\usepackage{arydshln}
\usepackage[hypertexnames=false,linkcolor=purple, colorlinks=true, citecolor=purple, filecolor=purple,urlcolor=purple,pagebackref]{hyperref}

\usepackage{marvosym}

\usepackage{tikz}

\usepackage{algorithmicx}
\usepackage{algorithm}
\usepackage{algpseudocode}

\algtext*{EndWhile}
\algtext*{EndIf}
\algtext*{EndFor}

\usepackage{eqparbox}

\usepackage{float}
\newfloat{algorithm}{t}{lop}

\usepackage{tikz}
\usetikzlibrary{shapes,arrows}
\tikzstyle{decision} = [diamond, draw,
    text width=5em, text centered, node distance=3cm, inner sep=5pt]
\tikzstyle{block} = [rectangle, draw,
    text width=8em, text centered, rounded corners, minimum height=4em]
\tikzstyle{line} = [draw, -latex']
\tikzstyle{cloud} = [draw, ellipse, node distance=3cm,
    minimum height=2em]

\usepackage{amsfonts,amssymb,amsthm,amscd,epsfig,epstopdf,url,array}

\algnewcommand\algorithmicinput{\textbf{INPUT:}}
\algnewcommand\INPUT{\item[\algorithmicinput]}

\algnewcommand\algorithmicoutput{\textbf{OUTPUT:}}
\algnewcommand\OUTPUT{\item[\algorithmicoutput]}

\newcounter{CountAlpha}

\theoremstyle{theorem}

\newtheorem{MainThm}[CountAlpha]{Theorem}  

\theoremstyle{plain}
\newtheorem{Thm}{Theorem}[section]

\newtheorem{Lem}[Thm]{Lemma}

\theoremstyle{definition}
\newtheorem{Def}[Thm]{Definition}

\newtheorem{Rk}[Thm]{Remark}
\newtheorem{Ex}[Thm]{Example}

\theoremstyle{remark}
\newtheorem*{acknowledgements}{Acknowledgements}

\newcommand{\CC}{{\mathbb{C}}}
\newcommand{\FF}{{\mathbb{F}}}

\newcommand{\ZZ}{{\mathbb{Z}}}

\newcommand{\cA}{{\mathcal{A}}}
\newcommand{\cI}{{\mathcal{I}}}
\newcommand{\cJ}{{\mathcal{J}}}
\newcommand{\cM}{{\mathcal{M}}}

\newcommand{\Bl}{{B\ell}} 
\newcommand{\Spec}{\operatorname{Spec}}
\newcommand{\Proj}{\operatorname{Proj}}
\newcommand{\car}{\operatorname{char}}

\newcommand{\Sing}{\operatorname{Sing}}

\newcommand{\pf}{\operatorname{pf}}

\begin{document}
\title [Desingularization of generic (skew-) symmetric determinantal singularities
]
{
Desingularization of generic symmetric and generic skew-symmetric determinantal singularities
}

\let\thefootnote\relax\footnotetext{Sabrina Alexandra Gaube: 
	Institut f\"ur Mathematik, 
	Carl von Ossietzky Universit\"at Oldenburg, 
	26111 Oldenburg, Germany.
	\\
 	email: sabrina.gaube@uni-oldenburg.de}

\author[Gaube]{Sabrina Alexandra Gaube}

\footnotetext{Bernd Schober ({\large\Letter}):
	Institut f\"ur Mathematik, 
	Carl von Ossietzky Universit\"at Oldenburg, 
	26111 Oldenburg, Germany.
	\\
	{\em Current address:} None. (Hamburg, Germany).
	\\
	email: schober.math@gmail.com}

\author[Schober]{Bernd Schober}

\thanks{B.S.~was partially supported by the project 
	``Order zeta functions and resolutions of singularities" funded by the Deutsche Forschungsgemeinschaft (DFG) 
	(DFG project number: 373111162).}


\makeatletter
\@namedef{subjclassname@2020}{\textup{2020} Mathematics Subject Classification}
\makeatother


\begin{abstract}
	We discuss how to resolve generic skew-symmetric  and generic symmetric determinantal singularities. The key ingredients are (skew-) symmetry preserving matrix operations in order to deduce an inductive argument.
\end{abstract}

\maketitle

\section{Introduction}

The class of determinantal singularities generalizes the class of complete intersections and carries more structure than arbitrary singularities because of the underlying matrix structure. 
For references providing more on the background of determinantal singularities, we refer to
the survey of Fr\"uhbis-Kr\"uger and Zach \cite{FKZ}, the book of Bruns and Vetter \cite{BV}, or the book of Harris 
\cite[Lecture~9]{Harris}.
More specialized approaches are the studies of symmetric determinantal singularities, as in the articles of Gaffney and Molino \cite{Gaffney1} and \cite{Gaffney2} or the studies of skew-symmetric determinantal singularities by Bruce, Goryunov and Haslinger \cite{BGH}, 
for example.

In this article, we focus on resolution of generic skew-symmetric resp.~generic symmetric determinantal singularities.
A local variant of resolution via blowing ups of regular centers of more general determinantal singularities, namely such determinantal singularities for which there exists a procedure which simultaneous locally monomializes the initial entries (e.g. binomial entries), will be discussed in \cite{Gaube}, see also \cite{SabrinaThesis}.

\medskip 

We fix positive integers $ m, r \in \ZZ_+ $ with $ r \leq m $. 
Let $ R_0  $ be a commutative regular ring
(e.g., $ R_0 = \CC, \FF_q, \ZZ, \ZZ[T]_{\langle 2 , T \rangle }, \ldots $)
and let 
\[
	R := R_0 [x_{i,j} \mid 1 \leq i \leq j  \leq m ]
\]
be the polynomial ring with $ m^2 $ independent variables and coefficients in $ R_0 $.
Consider a generic square matrix
(In fact, more generally, one may take matrices into account that are not necessarily square matrices.
Since the (skew-) symmetric setting is guiding us to consider square matrices, we neglect the mentioned more general situation here.)
\[
	M_{m} := (x_{i,j})_ {i,j} =  
		\begin{pmatrix}
			x_{1,1}  &  x_{1,2} & \cdots     & x_{1,m}
			\\
			x_{2,1}  &  x_{2,2} & \cdots     & x_{2,m}
			\\
			\vdots & \vdots & \ddots  & \vdots    
			\\ 
			x_{m,1}  &  x_{m,2} & \cdots     & x_{m,m}
		\end{pmatrix}
		.
\]
In \cite[Section 5]{BerndPartial}, the second named author discussed how to resolve the singularities of the
generic determinantal singularity $ X_{m,r} $ 
that is defined by the vanishing of all $ r $-minors of $ M_m $,
i.e., $X_{m,r} := \Spec(R/\cJ_{m,r})$ for 
\[
	\cJ_{m,r} := \langle 
	f_{I,J} \mid I, J \subseteq \{ 1, \ldots, m \} : \# I = \# J = r \rangle
	\subset R, 
\] 
where 
$ 
	f_{I,J} := \det (M_{I,J}), 
$ 
for	$I, J \subseteq \{ 1, \ldots, m \},\# I = \# J = r$ and where $ M_{I,J} := (x_{i,j})_{i\in I, j \in J} $ denotes the $ r \times  r $ submatrix of $ M_m $ determined by $ I $ and $ J $.
This generalized a result by Vainsencher \cite{DetermResol} who provided a desingularization for generic determinantal singularities if $ R_0 $ contains an algebraic closed field.

It is a natural question to ask whether a desingularization exists for non-generic determinantal singularities which are equipped with some additional structure, i.e., where the entries $ x_{i,j} $ are replaced by more general elements in $ R $.
Furthermore, it is an interesting problem to investigate to what extent the determinantal structure resp.~the geometry of the entries contribute to the singularity and its resolution.
If $ R_0 $ contains a field of characteristic zero, there exists a resolution by Hironaka \cite{Hiro64}.
First constructive proofs have been given by Bierstone and Milman \cite{BM} and by Villamayor \cite{Vil1,Vil2}; 
we also refer to \cite{Cu,EV,Kollar}, for example.
We allow $ R_0 $ to be of positive and mixed characteristics. 
Hence, the question on the existence of a desingularization is neither known nor clear, in general.

In the present article, we provide an affirmative answer to this in the generic symmetric and generic skew-symmetric setup, where we consider the generic matrix with additional relations $x_{i,j} = x_{j,i}$ resp.~$x_{i,j} = -x_{j,i}$ for all $i,j \in \{1,\ldots,m\}$.
A different approach for not necessarily generic entries in the matrix will appear in \cite{Gaube}, where the determinantal structure takes an essential role, see also \cite{SabrinaThesis}.

\medskip 

Let us fix some notation in order to formulate our main result in the skew-symmetric  case.
First, assume $\car(R_0) \neq2$. Then $x_{i,j} = -x_{j,i}$ implies for $i = j$ that $x_{i,i} = 0$. 
Hence we introduce
\[
	I_{\rm skew} := \langle x_{i,j} + x_{j,i}, \, x_{k,k} \mid   
	1 \leq i < j \leq m, 1 \leq k \leq m
	\rangle
\]	
and define 
\[
	S_{\rm skew} := R/I_{\rm skew},  
	\hspace{10pt} 
	Z_{\rm skew} := \Spec(S_{\rm skew}),
	\hspace{10pt}
	\cI^{\rm skew}_{m,r} := \cJ_{m,r} + I_{\rm skew} ,
\] 
\[  
	Y^{\rm skew}_{m,r} := \Spec ( R / \cI^{\rm skew}_{m,r} ) \subseteq Z_{\rm skew}. 
\]
In other words, 
if $ \car(R_0) \neq 2 $, 
then
$ Y^{\rm skew}_{m,r} $ is defined by the vanishing locus of the $ r $-minors of the skew-symmetric matrix
\begin{equation}
\label{eq:Am}
	A_{m} :=  
	\begin{pmatrix}
	0  &  x_{1,2} & \cdots     & x_{1,m}
	\\
	-x_{1,2}  &  0 &      & x_{2,m}
	\\
	\vdots &  & \ddots  & \vdots    
	\\ 
	-x_{1,m}  &  -x_{2,m} & \cdots     & 0
	\end{pmatrix}
	.
\end{equation} 
Recall the following facts:
\begin{itemize}
	\item[(F1)] 
	Since $ \car(R_0) \neq 2 $, it is easy to show that the determinant of every skew-symmetric matrix of odd size is identically zero. 
	In particular, $ Y^{\rm skew}_{2\ell-1,2\ell-1} = Z_{\rm skew} $ is regular for every $ \ell \in \ZZ_+ $.
	
	\item[(F2)] 
	By \cite[Theorem 3.1.(1)]{KLS}, $\sqrt{\langle 2\ell\text{-minors of }A_m \rangle} = \sqrt{\langle (2\ell-1)\text{-minors of }A_m \rangle}, $ where $\sqrt{I}$ denotes the radical of an ideal $I$ and $ \ell \in \ZZ_+ $.
	
	\item[(F3)] 
	The equality
	$ \det(A) = (\pf(A))^2 $ holds for any skew-symmetric square matrix $A$,
	where $ \pf(A) $ denotes the pfaffian of $ A $.
	\\
	Recall that a principal minor of a matrix is a minor for which the index set of the corresponding columns and rows coincide.
	Its follows from \cite[Theorem~~3.12(2) and (3)]{KLS} that we have 
	$\sqrt{\langle 2\ell\text{-minors of }A_m \rangle}
	=
	\sqrt{\langle \text{principal-}2\ell\text{-minors of }A_m \rangle} 
	= 
	\sqrt{\langle \text{pfaffian of principal-}2\ell\text{-minors of }A_m \rangle} $. 
\end{itemize}

\noindent 
By (F2), it is sufficient to consider only the vanishing loci determined by the $2\ell$-minors in the skew-symmetric setting.
Writing $ X_{\rm red} $ for the reduction of $ X $, 
for every $ \ell \in \ZZ_+ $, we have
\begin{equation}
	\label{eq:2l_2l-1}
	( Y^{\rm skew}_{m,2\ell} )_{\rm red} 
	\cong 
	( Y^{\rm skew}_{m,2\ell - 1} )_{\rm red}.
\end{equation}
Using (F3), we could further reduce this to the vanishing loci of the pfaffian of the principal-$ 2 \ell $-minors of $ A_m $.
Since we will mainly work with the transform of the matrix $ A_m $ along a blowing up, we stick to the $ 2 \ell $-minors.    
For more details on pfaffians and the ideals generated by them, we refer to \cite{KLS}.

\smallskip 

An {\em embedded resolution of singularities} of a reduced scheme $ Y \subset Z $ with $ Z $ regular, is a proper, birational morphism $ \pi : Z' \to Z $, for some regular scheme $ Z ' $,
such that 
\begin{itemize}
\item[(a)] the strict transform $ Y' $ of $ Y $ in $ Z '$ is regular,
\item[(b)] $ \pi $ is an isomorphism outside of the singular locus of $ Y $, 
i.e., $ \pi^{-1} (Z \setminus \Sing(Y)) \cong Z \setminus \Sing(Y) $,
\item[(c)] $ \pi^{-1}(\Sing(Y)) $ is a simple normal crossings divisor
(i.e., all irreducible components are regular and they intersect transversally) 
which intersects $ Y' $ transversally. 
\end{itemize}
One way to achieve this is to construct a finite sequence of blowing ups in regular centers contained in the singular locus of the strict transform of $ Y $.

\medskip 

It is necessary to be careful and to work in the reduced setting.
For example, $ Y^{\rm skew}_{2,2} $ is defined by the vanishing of 
$ \det
\begin{pmatrix}
	0 & x_{1,2}\\
	-x_{1,2} & 0 
\end{pmatrix}  
=  x_{1,2}^2 $.
Hence, $ Y^{\rm skew}_{2,2} $ is 
non reduced, which implies that no blowing up in a regular center improves the singularities, while on the other hand its reduction is regular.

\begin{MainThm}
	\label{Thm_1}
	Let $ m, \ell \in \ZZ_+ $ be positive integers with $ 2\ell \leq m$,
	let $ R_0 $ be a commutative regular ring with $\car(R_0) \neq 2$.
	%
	The following sequence of blowing ups is an embedded resolution of singularities for the reduction of the generic skew-symmetric determinantal singularity $ Y^{\rm skew}_{m,2\ell} \subset Z_{\rm skew} $,
	\[
		Z_{\rm skew} =: Z_0 \stackrel{\pi_1}{\longleftarrow} 
		Z_{1} \stackrel{{\pi_2}}{\longleftarrow}  		
		\ldots 
		\stackrel{\pi_{{\ell -1}}}{\longleftarrow} 
		Z_{\ell - 1},
	\]
	where $ \pi_{\alpha} \colon Z_\alpha \to Z_{\alpha- 1} $ is the blowing up with center the strict transform of $ {(Y^{\rm skew}_{m,2 \alpha})_{\rm red}} \cong {(Y^{\rm skew}_{m,2 \alpha-1})_{\rm red}} $ in $ Z_{\alpha-1 } $, 
	for $ \alpha \in \{ 1, \ldots, \ell - 1 \}$.
\end{MainThm}

The situation changes slightly if we move on to $ \car(R_0) = 2 $. 
The relation $x_{i,j} = -x_{i,j}$ becomes $x_{i,j} = x_{j,i}$. 
In particular, for $i = j$ the relation $x_{i,j} = -x_{j,i}$ does not imply $x_{i,i} = 0$. 
So if $\car(R_0)=2$, the skew-symmetric  generic case naturally leads us to the generic symmetric case.

\medskip 

Let us introduce the symmetric setup. 
	Let $ R_0 $ be a regular ring of arbitrary characteristic.
We define
\[ 
	I_{\rm sym} := 
	\langle x_{i,j} - x_{j,i} \mid   
	1 \leq i < j \leq m
	\rangle,
\]	

\[ 
	S_{\rm sym} := R/I_{\rm sym}, 
	\hspace{10pt}
	Z_{\rm sym}:= \Spec(S_{\rm sym}),
	\hspace{10pt} 
	\cI_{m,r}^{\rm sym}:= \cJ_{m,r} + I_{\rm sym}, 
\]

\[ 
	Y_{m,r}^{\rm sym} :=\Spec(R/\cI^{\rm sym}_{m,r}) \subset Z_{\rm sym}. 
\]

\bigskip 

\noindent 
Therefore, $Y_{m,r}^{\rm sym}$ is 
given as the locus where all $ r $-minors of the following matrix $ B_m $ vanish,
\begin{equation}
\label{eq:Bm}
B_m := 
\begin{pmatrix}
x_{1,1}  &  x_{1,2} & \cdots     & x_{1,m}
\\
x_{1,2}  &  x_{2,2} &      & x_{2,m}
\\
\vdots &  & \ddots  & \vdots    
\\ 
x_{1,m}  &  x_{2,m} & \cdots     & x_{m,m}
\end{pmatrix}.
\end{equation}

\medskip 

\begin{MainThm}
	\label{Thm_2}
	Let $ m, r \in \ZZ_+ $ be positive integers with $ r \leq m $,
	let $ R_0 $ be a commutative regular ring. 
	%
	The following sequence of blowing ups is an embedded resolution of singularities for the generic symmetric determinantal singularity $ Y_{m,r}^{\rm sym} \subseteq Z_{\rm sym} $,
	\[
		Z_{\rm sym} =: Z_0 \stackrel{\pi_1}{\longleftarrow} 
		Z_{1} \stackrel{\pi_2}{\longleftarrow} 
		\ldots 
		\stackrel{\pi_{r-1}}{\longleftarrow} 
		Z_{r-1}
	\]
	where $  \pi_{\alpha} \colon Z_\alpha \to Z_{\alpha- 1} $ is the blowing up with center the strict transform of $ Y_{m,\alpha}^{\rm sym} $ in $ Z_{\alpha - 1 } $, for $ \alpha \in \{  1 , \ldots , r - 1 \} $.
\end{MainThm}

The proofs of Theorems~\ref{Thm_1} and~\ref{Thm_2} are done by induction on the size $ m \geq 1 $ of the matrix defining the determinantal singularity.
If $ m = 1 $, there is nothing to show and the case $ m = 2 $ is easy (Example~\ref{Ex:22}). 
For $ m > 2 $, the key ingredient is to apply suitable row and column operations to the strict transform of $A_m$ resp.~$B_m$ which preserve the (skew-) symmetric structure in each chart of a blowing up.
Eventually, the determinantal structure allows to make a reduction to cases, where the respective integer $ m $ decreased. 

\smallskip

Based on Theorems~\ref{Thm_1} and \ref{Thm_2}, 
the first named author studied aspects of an implementation as well as the complexity of the constructed desingularizations.
In particular, see 
\cite[Section~A.5.1 and Section~B.2]{SabrinaThesis}
and \cite[Chapter~8]{SabrinaThesis}, where algorithmic aspects for more general determinantal singularities are considered. 

\smallskip

Let us briefly summarize the content of the paper: In Section~\ref{Sec:Ex},
we discuss examples for small values of $ m $ to visualize the ideas 
and to provide the basis for our induction. 
In Section~\ref{sec:proof1}, we prove Theorem~\ref{Thm_1} and in Section~\ref{sec:proof2}, we show Theorem~\ref{Thm_2}.

All rings considered in this article are assumed to be commutative.

\section{First Examples}
\label{Sec:Ex}
First, we fix some notation and illustrate the ideas and appearing phenomena for some examples.

\begin{Rk}\label{Rk:blowup}
Let $ R_0 $ be a regular ring and $ Z = \Spec(R_0[t_1, \ldots, t_N]) \cong \mathbb{A}^N_{R_0} $. 
Consider the blowing up $ \pi \colon \Bl_D(Z) \to Z $
of $ Z $ with center $ D = \Spec(R_0[t_1, \ldots, t_N]/ \langle t_i \mid i \in I \rangle ) $,
for some $ I \subseteq \{ 1, \ldots, N \} $.
Set $ d := \# I $.
Globally, $ \Bl_D(Z) $ is described by the $ \Proj $-construction as 
\[
	\Bl_D(Z) = \Proj (R_0 [t_1, \ldots, t_N][T_i \mid i \in I ]/ \langle  t_i T_j - t_j T_i \mid i, j \in I \rangle )
	\subseteq 
	\mathbb{A}^N_{R_0} \times \mathbb{P}^{d-1}_{R_0} ,
\]
where $ (T_i \mid i \in I) $ are projective coordinates,
and $ \pi $ is induced by the projection on the first factor $ 
\mathbb{A}^N_{R_0} \times \mathbb{P}^{d-1}_{R_0}  \to 
\mathbb{A}^N_{R_0} $. 
In particular, $ \Bl_D(Z) $ is covered by the affine charts $ D_+ (T_i) $
with $ i \in I $,
where $ D_+ (T_i) $ denotes the standard open set of points where $ T_i $ is invertible. 
We call $ D_+ (T_i) $ also the {\em $ T_i $-chart} of the blowing up. 

Fix $ i \in I $. 
In the $ T_i $-chart, the relation $ t_i T_j - t_j T_i = 0  $ can be rewritten as 
$
	t_j = t_i \frac{T_j}{T_i},
$
for 
$ j \in I \setminus \{ i \} $.
This provides that $ \Bl_D (Z) \cap D_+ (T_i) $ is isomorphic to 
\[ 
	\Spec(R_0[t_k, t_i, \frac{T_j}{T_i} \mid k \notin I, j \in I \setminus \{i\} ]) 
	\cong \mathbb{A}_{R_0}^{N} .
\]
Often the abbreviation $ t_j ' := \frac{T_j}{T_j} $ is used and by setting $ t_i' := t_i $ and $ t_k' := t_k $ for $ k \notin I $,
it is said that the variables of the $ T_i $-chart are $ (t_1', \ldots, t_N') $, 
where the relation to the variables $ (t_1, \ldots, t_N) $ before the blowing up is described by
\begin{equation}
	\label{eq:trans_rule}
	t_j = \begin{cases}
	t_i' t_j',  & \mbox{if } j \in I \setminus \{ i \}, \\
	t_j', 	& \mbox{otherwise.}
	\end{cases}
\end{equation} 

Note that the exceptional divisor of the blowing up $ \pi $,
i.e., the preimage of $ D $ along $ \pi $
(the locus where $ \pi $ is not an isomorphism), 
is given by the divisor $ \operatorname{div}(t_i') $ in the $ T_i $-chart. 

Finally, recall that for an element $ f \in R_0 [t_1, \ldots, t_N] $, its {\em strict transform} in the $ T_i $-chart is the element $ f' \in R_0 [t_1', \ldots, t_N' ] $
which we obtain by applying the transformation rule \eqref{eq:trans_rule} to $ f $ and then factoring the exceptional variable $ t_i' $ as much as possible. 
On the other hand, given $  X  \subseteq Z $ its {\em strict transform $ X' \subseteq Z' $} is the closure of $ \pi^{-1}(X \setminus D) $ in $ Z' $.  

\smallskip 

In our setting, we have $ (t_1, \ldots, t_N) = (x_{i,j} \mid i, j \in \{ 1, \ldots, m\} ) $ and we speak of the $ X_{1,2} $-chart, or similar expressions.
Besides reflecting the origin of the variables $ (x_{i,j} ) $ coming from a matrix structure, the index set $ \{ 1, \ldots, m \}^2 $ has no impact or deeper meaning.
\end{Rk}

\begin{Ex}[Theorems~\ref{Thm_1} and \ref{Thm_2} for $ m = 2 $]
	\label{Ex:22}
	Fix a regular basis ring $ R_0 $.
		Let us consider $ Y_{2,r}^{\rm skew} $ and $ Y_{2,r}^{\rm sym} $ for $ r \in \{ 1, 2 \} $.
	\begin{enumerate}
		\item 
		It is clear that both are regular for $ r = 1 $ since we assume $ R_0 $ to be regular.
	
		\item
		Suppose that $ \car(R_0) \neq 2 $.
		We have already seen in the introduction 
		that the reduction of $ Y_{2,2}^{\rm skew} $ is regular and hence no further blowing ups are required.
		
		\item 
		In the symmetric case, $ Y_{2,2}^{\rm sym} $ is a singular hypersurface given by the equation
		\[
			\det \begin{pmatrix}
			x_{1,1} & x_{1,2} \\
			x_{1,2} & x_{2,2} 
			\end{pmatrix}
			= x_{1,1} x_{2,2} - x_{1,2}^2 = 0 .
		\] 
		The singular locus is equal to $ Y_{2,1}^{\rm sym} $ and by blowing up the center $ D_1 = Y_{2,1}^{\rm sym} = V(x_{1,1},x_{1,2,},x_{2,2}) $, 
		we obtain an embedded desingularization. 
	\end{enumerate}
	Therefore, we have verified Theorems~\ref{Thm_1} and~\ref{Thm_2} for the special case $ m = 2 $.
\end{Ex}

In general, $ Y_{m,r}^{\rm skew} $ and $ Y_{m,r}^{\rm sym} $ are not hypersurfaces if $ r < m $.
In order to be able to control them after a blowing up, we need the following result.

\begin{Lem}
	\label{Lem:Trafo}
	Let $ R_0 $ be a regular ring and 
	$ X = \Spec (R_0 [t_1, \ldots, t_N] / I ) = Z \cong \mathbb{A}^N_{R_0} $,
	for some non-zero ideal $ I \subset R_0 [t_1, \ldots, t_N] $.
	Consider the blowing up $ \pi \colon Z' \to  Z $ with center $ D = \Spec(R_0 [t_1, \ldots, t_N] / \langle t_1, \ldots, t_N \rangle ) $.
	If there exists a system of generators $ (f_1, \ldots, f_m) $ for $ I $ such that each $ f_j $ is homogeneous in the variables $ (t_1, \ldots, t_N ) $,
	then the ideal of the strict transform of $ X $ in the $ T_i $-chart is generated by the strict transforms of $ (f_1, \ldots, f_m) $.   
\end{Lem}

In general, the hypothesis to have homogeneous generators is necessary.
For example, consider 
$ R = K[x,y,z] $, $ P = \langle x, y, z \rangle \subset R $,  
and the ideal
$ I = \langle x^2 -y^3,x^2-z^5\rangle \subset R $, where $K$ is any field. 
Clearly, $ (g_1, g_2) := (x^2 - y^3 , x^2 - z^5 ) $ are not homogeneous with respect to $ P $.
Set $ h := y^3 - z^5 = g_2 - g_1 $.
If we blow up the origin $ D = \Spec (R/P) $ and consider the $ Z $-chart, we get the strict transforms
$ g_1' = x'^2 - y'^3 z' $,
$ g_2' = x'^2 - z'^3  $,
$ h' = y'^3 - z'^2 $.
Note that $ g_2' - g_1' = z' (y'^3 - z'^2 ) $. 
Let $ I' $ be the ideal of the strict transform $ X' $ in the given chart. 
We see that $ h' \in I' $,
but $ h' \notin \langle g_1', g_2' \rangle $.
Hence, $ I' \neq \langle g_1' , g_2' \rangle $.

\begin{proof}[Proof of Lemma~\ref{Lem:Trafo}]
	Suppose the conclusion is false. 
	Then, there exists $ h \in I = \langle f_1, \ldots, f_m \rangle $ such that its strict transform $ h' $ is not contained in $ \langle f_1', \ldots, f_m' \rangle $.
	Set $ P := \langle t_1, \ldots, t_N \rangle $ and  
	$ r_i := \operatorname{ord}_P (f_i) := 
	\sup \{ k \geq 0 \mid f_i \in P^k \} $, for $ i \in \{ 1, \ldots, N \} $.

	We expand $ h $ as finite sum $ h = \sum_{i=1}^m \sum_{A\in \ZZ_{\geq 0}^N} \lambda_{i,A} t^A f_i $ with coefficients $ \lambda_{i,A} \in R_0 $.
	We have $ \operatorname{ord}_P (h) \geq \min \{ \operatorname{ord}_P (t^A f_i) \mid \lambda_{i,A} \neq 0  \} =: \mu $ and if equality holds, then it is easy to verify that $ h' \in \langle f_1', \ldots, f_m' \rangle $ by applying the rule of transformation \eqref{eq:trans_rule}.
	Thus, this would lead to a contradiction. 
		
	On the other hand, $ \operatorname{ord}_P (h) > \mu $ implies that there is cancellation within the part of the expansion which is homogeneous of degreee $ \mu $ in $ (t) $, i.e.,
	$ \sum_{i=1}^m \sum_{A : |A|+r_i = \mu} \lambda_{i,A} t^A f_i \equiv 0 $,
	where $ |A| := A_1 + \ldots + A_N $ for $ A = (A_1, \ldots, A_N) $. 
	Notice that the hypothesis that all $ f_i $ are homogeneous is used here, as the degree $ \mu $-part does not lie in the ideal $ \langle f_1, \ldots, f_m \rangle $, in general. 
	But now, we may omit the degree $ \mu $-part of the expansion, as it is obsolete, and repeat our argument for the new expansion.
	Since $ h $ has to be non-zero and the expansion is finite, we eventually end up in the first case, which provides a contradiction, as desired.
\end{proof}

In order to discuss the relevant examples for the lemma in our context and for the later use, we introduce the following notation.

\begin{Def}
	\label{Def:IrM}
	Let $ r, m \in \ZZ_+ $ with $ r \leq m $.
	Let $ R $ be a regular ring and let $ M = (m_{i,j}) $ be an $ m \times m $ matrix with entries in $ R $. 
	We define $ \cI_r (M) \subseteq R  $ to be the ideal generated by the $ r $-minors of $ M $,
	i.e.,
	\[
		\cI_r (M) = \langle \det( M_{I,J} ) \mid I, J \subseteq \{ 1, \ldots, m \} : \# I = \# J = r \rangle  ,
	\]
	where $ M_{I,J} $ denotes the $ r \times r $ submatrix of $ M $ given by $ (m_{i,j})_{i\in I, j \in J } . $
\end{Def}

\medskip 

We fix a regular ring $ R_0 $ for the remainder of the section.

\begin{Ex}
	\label{Ex:std}
	Let $ r, m \in \ZZ_{+} $ with $ r \leq m $ and consider the matrices $ A_m $ and $ B_m $ introduced in \eqref{eq:Am} and \eqref{eq:Bm}, respectively.
	Using the notation of the introduction, we have 
	\[  
		\cI_{m,r}^{\rm skew} = \cI_r (A_m) 
		\hspace{.5cm}
		\mbox{ and }
		\hspace{.5cm}
		\cI_{m,r}^{\rm sym} = \cI_r (B_m) 
	\]
	in $ R_0 [x_{i,j} \mid 1 \leq i < j \leq m ] $ (with $ \car(R_0) \neq 2 $), resp.~$ R_0 [x_{i,j} \mid 1 \leq i \leq j \leq m ] $.
	
	Let $ I, J \subseteq \{ 1, \ldots, m \} $ be subsets of cardinality $ r $.
	Set $ f_{I,J}^{{\rm skew}} := \det((A_m)_{I,J}) $ and $ f_{I,J}^{{\rm sym}} := \det((B_m)_{I,J}) $. 
	Clearly, $ f_{I,J}^{{\rm skew}} $ and $ f_{I,J}^{{\rm sym}} $ are homogeneous of degree $ r $ in the variables $ (x_{i,j}) $.
	Therefore, these elements provide systems of generators for the ideal $ \cI_{m,r}^{\rm skew} $, resp.~$ \cI_{m,r}^{\rm sym} $, fulfilling the hypothesis of Lemma~\ref{Lem:Trafo}.
\end{Ex}

In the next example, we illustrate how we preserve the symmetry by appropriate row and column operations.

\begin{Ex}[Theorem~\ref{Thm_2} for $ m = 3 $]
	\label{Ex:33sym}
	Consider $ Y_{3,3}^{\rm sym} $, which is determined by the vanishing of the determinant of $ B_3 $,
	\[
	B_{3} = 
	\begin{pmatrix}
	x_{1,1}  &  x_{1,2} & x_{1,3} 
	\\
	x_{1,2}  &  x_{2,2} &  x_{2,3} 
	\\
	x_{1,3}  &  x_{2,3} & x_{3,3} 
	\end{pmatrix}.
	\]
	By Theorem~\ref{Thm_2},
	the first center to blow up
	is 
	$ 
		Y^{\rm sym}_{3,1} = \Spec (S_{\rm sym}/ \langle x_{1,1} , x_{1,2} , \ldots, x_{3,3} \rangle ) .
	$ 
	Our strategy for the next step is to check the following two things after the first blowing up:
	\begin{enumerate}
		\item 
		All singularities of the strict transform of $Y_{3,3}^{\rm sym} $ lie on the $X_{i,i}$-charts, $i \in \{ 1,2,3 \} $;
		
		\item
		The strict transform of $Y_{3,2}^{\rm sym}$ provides the next center to blow up. 
	\end{enumerate}

	Consider the $ X_{1,1}$-chart,
	i.e., as described in Remark~\ref{Rk:blowup},
	we have $ x_{1,1} = x_{1,1}' $ and 
		$ x_{i,j} = x_{1,1}' x_{i,j}' $  for $ (i,j) \neq (1,1) $.
	The strict transform of $ Y_{3,3} $ is defined by the vanishing locus of the determinant of
	\[ 
	B_{3}' := 
	\begin{pmatrix}
	1  &  x_{1,2}' & x_{1,3}' 
	\\
	x_{1,2}'  &  x_{2,2}' &  x_{2,3}'
	\\
	x_{1,3}'  &  x_{2,3}' & x_{3,3}' 
	\end{pmatrix}.
	\]
	We perform elementary row and column operations to eliminate all entries  in the first row and column expect for the $ 1 $ at position $ (1,1) $.
	Since these do not change the determinant, we have
	\[ 
	\det (B_{3}') = 
	\det
	\begin{pmatrix}
	1  &  0 & 0
	\\
	0  &  x_{2,2}'- x_{1,2}'^2 &  x_{2,3}'-x_{1,2}'x_{1,3}'
	\\
	0  &  x_{2,3}'-x_{1,2}'x_{1,3}' & x_{3,3}' - x_{1,3}'^2
	\end{pmatrix}
	 = 
	 \det
	 \begin{pmatrix}
	 y_{2,2} &  y_{2,3}
	 \\
	 y_{2,3} & y_{3,3}
	 \end{pmatrix},
	\]
	where 
	$ y_{2,2} := x_{2,2}' - x_{1,2}'^2 $,
	$ y_{2,3} := x_{2,3}' - x_{1,2}'x_{1,3}' $,
	$ y_{3,3} := x_{3,3}' - x_{1,3}'^2 $.
	Recall that the exceptional divisor of the first blowing up is $ \operatorname{div}(x_{1,1}') $ and $ x_{1,1}' $ is transveral to $ y_{2,2}, y_{2,3}, y_{3,3} $.

	We have reduced the problem to $ r = m = 2 $ and the ideal of the upcoming center is $ I := \langle y_{2,2}, y_{2,3}, y_{3,3} \rangle $. 
	This blowing up resolves the singularities, as mentioned in Example~\ref{Ex:22}(3). 
	
	We claim that the ideal $ I $ provides a global center. 
	More precisely, in the following it is shown that the strict transform of $Y_{3,2}^{\rm sym} $ is given by a sheaf of ideals which in the $ X_{1,1} $-chart is given by $ I $:
	We observe that $ y_{i,j} $ is the strict transform of the $ 2 $-minor of $ B_3 $ that contains $ x_{1,1} $ and $ x_{i,j} $,
	\[
	\det \begin{pmatrix}
	x_{1,1} & x_{1,j} \\ x_{1,i} & x_{i,j}
	\end{pmatrix}
	= x_{1,1} x_{i,j} - x_{1,j} x_{1,i}.
	\]
	The strict transform of every $2$-minor of $B_3$ that does not contain $x_{1,1}$ is already contained in $I$:
	For $ i, j, \ell \in \{ 2, 3 \} $, we have
	\[ 
		\det \begin{pmatrix}
			x_{1,\ell}' &  x_{i,\ell}' \\ x_{1,j}'  & x_{i,j}'
		\end{pmatrix}
		=
		\det \begin{pmatrix}
			x_{1,\ell}' & y_{i,\ell} + x'_{1,i} x'_{1,\ell} \\ x_{1,j}'  & y_{i,j} + x'_{1,i} x'_{1,j}
		\end{pmatrix}
		=
		\det \begin{pmatrix}
			x_{1,\ell}' & y_{i,\ell} \\ x_{1,j}'  & y_{i,j} 
		\end{pmatrix} 
		\in I ,
	\]
	where we apply an elementary column operation for the second equality.
	Further, we get
	
	\[ 
	\det \begin{pmatrix}
		x_{2,2}' &  x_{2,3}' \\ x_{2,3}'  & x_{3,3}'
	\end{pmatrix}
	=
	\det \begin{pmatrix}
		y_{2,2} + x_{1,2}'^2 & y_{2,3} + x'_{1,2} x'_{1,3} \\ y_{2,3} + x'_{1,2} x'_{1,3}  & y_{3,3} + x_{1,3}'^2 
	\end{pmatrix}
	=
	\]

	\[
	=
	\det \begin{pmatrix}
		1 & 0 & 0 \\
		x_{1,2}' & y_{2,2} + x_{1,2}'^2 & y_{2,3} + x'_{1,2} x'_{1,3} \\ x_{1,3}' & y_{2,3} + x'_{1,2} x'_{1,3}  & y_{3,3} + x_{1,3}'^2 
	\end{pmatrix}
	=
	\det \begin{pmatrix}
		1 & -x_{1,2}' & -x_{1,3}' \\
		x_{1,2}' & y_{2,2}  & y_{2,3}  \\ 
		x_{1,3}' & y_{2,3}   & y_{3,3}  
	\end{pmatrix}
	.
	\]
	
	\noindent 
	In order to see that this is contained in $ I $, expand the last determinant with respect to the first row.

	In other words, the center of the second blowing up coincides with the strict transform of the determinantal singularity $ Y^{\rm sym}_{3,2} $ defined by the $ 2 $-minors of $ B_3 $ as suggested by Theorem~\ref{Thm_2}. 
	The situation in the $X_{2,2}$- and the $X_{3,3}$-chart is analogous.
	
	\smallskip 
	
	Let us study the $ X_{2,3} $-chart.
	(The cases of the $ X_{1,2} $- resp.~the $ X_{1,3} $-chart are analogous.)
	In order to have a clear distinction to the $ X_{1,1} $-chart, we write 
		$ \widetilde{x}_{i,j} $ instead of $ x_{i,j}' $ for the variables in the present chart. 
		Hence, the transformation of the variables is 
		$ x_{2,3} = \widetilde{x}_{2,3} $
		and 
		$ x_{i,j} =  \widetilde{x}_{2,3} \widetilde{x}_{i,j} $ for $ (i,j) \neq (2,3) $.
	The strict transform of $ Y_{3,3}^{\rm sym} $ is given by the determinant
	\[ 
	\det 
	\begin{pmatrix}
	\widetilde{x}_{1,1}  &  \widetilde{x}_{1,2} & \widetilde{x}_{1,3} 
	\\
	\widetilde{x}_{1,2}  &  \widetilde{x}_{2,2} &  1
	\\
	\widetilde{x}_{1,3}  &  1 &  \widetilde{x}_{3,3} 
	\end{pmatrix}
	=
	(\widetilde{x}_{2,2} \widetilde{x}_{3,3} - 1 ) \widetilde{x}_{1,1} + h,
	\]
	where $ h $ does neither depend on $ \widetilde{x}_{1,1} $ nor on the exceptional variable $ \widetilde{x}_{2,3} $.
	
	Using the notation of Remark~\ref{Rk:blowup},
	the set of points for which  
	$ \widetilde{x}_{2,2} \widetilde{x}_{3,3} -1 =  \frac{X_{2,2}}{X_{2,3}} \frac{X_{3,3}}{X_{2,3}} - 1 = 0 $ 
	is entirely contained in the intersection of the $ X_{2,2} $- and the $ X_{3,3} $-chart.
	Therefore, we may assume without loss of generality that $ \widetilde{x}_{2,2} \widetilde{x}_{3,3} - 1 $ is invertible in the given chart.
	This implies that the strict transform of $ Y_{3,3}^{\rm sym} $ is regular and transversal to the exceptional divisor $ \operatorname{div}(\widetilde{x}_{2,3}) $.
	In other words, we have resolved the singularities of $ Y_{3,3}^{\rm sym} $.
	
	Notice that $ \widetilde{x}_{2,2} \widetilde{x}_{3,3} - 1 $ is the strict transform of the $ 2 $-minor of $ B_3 $ containing $ x_{2,3} $ twice.
	Hence, the strict transform of $ Y_{3,2}^{\rm sym} $ is empty in the $ X_{2,3} $-chart, if we assume that we neglect the part that is already contained in $ D_+ (X_{2,2}) \cap D_+ (X_{3,3}) $.

	\smallskip 
	
	In conclusion, we have seen that the strict transform of $ Y_{3,2}^{\rm sym} $, 
	which is the next center for blowing up 
	proposed by Theorem~\ref{Thm_2}, is
	regular and it is contained in the union of three charts given by the diagonal elements,
	$ (Y^{\rm sym}_{3,2})' \subset D_+ (X_{1,1}) \cup D_+ (X_{2,2}) \cup D_+ (X_{3,3}) $.
	After blowing up $ (Y^{\rm sym} _{3,2})' $, we have resolved $ Y^{\rm sym}_{3,3} $. 
		
	{Note that} we have especially seen that $ Y^{\rm sym}_{3,2} $ is resolved after blowing up $ Y^{\rm sym}_{3,1} $.	
\end{Ex}

\medskip 

The next example illustrates the method for the reduction in the skew-symmetric setting.

\begin{Ex}[Theorem~\ref{Thm_1} for $ m = 4 $]
	\label{Ex:44skew}
	Assume that $ \car(R) \neq 2 $. 
	Recall that $ Y_{4,4}^{\rm skew} $, 
	is given by $ \det (A_4 ) = 0  $, where
\[
	A_{4} = 
	\begin{pmatrix}
	0  &  x_{1,2} & x_{1,3} & x_{1,4}
	\\
	-x_{1,2}  &  0 &  x_{2,3} & x_{2,4}
	\\
	-x_{1,3}  &  -x_{2,3} & 0 & x_{3,4}
	\\ 
	-x_{1,4}  &  -x_{2,4} & -x_{3,4}     & 0
	\end{pmatrix}.
\]
Theorem~\ref{Thm_1} suggests
$ Y^{\rm skew}_{4,1} = \Spec (S_{\rm skew}/ \langle x_{1,2}, \ldots, x_{3,4} \rangle) $
as the first center to blow up.

We consider the $ X_{1,2} $-chart of this blowing up. 
(All other charts are analogous.)
In there, 
we have $ x_{1,2} = x_{1,2}' $ and $ x_{i,j} = x_{1,2}' x_{i,j}' $ for $ (i,j) \neq (1,2) $.
	In particular, the strict transform
	$ (Y^{\rm skew}_{4,4})' $ is defined by $ \det (A_4') = 0 $,
	where
\[
	A_{4}' = 
	\begin{pmatrix}
	0  &  1 & x_{1,3}' & x_{1,4}'
	\\
	-1  &  0 &  x_{2,3}' & x_{2,4}'
	\\
	-x_{1,3}'  &  -x_{2,3}' & 0 & x_{3,4}'
	\\ 
	-x_{1,4}'  &  -x_{2,4}' & -x_{3,4}'     & 0
	\end{pmatrix}.
\]
Observe that $ ( Y^{\rm skew}_{4,1} )' = \varnothing $ (since the strict transform of the center is always empty) and $ ( Y^{\rm skew}_{4,2} )' = \varnothing $ (since there is a $ 2 $-minor that is equal to one after the blowing up).

We perform elementary row operations on $ A_4' $.
First, we add $x_{2,3}'$-times the first row to the third row and $x_{2,4}'$-times the first row to the fourth row. 
After that the second column becomes $ (1,0,0,0)^T $ and we can eliminate the entries at position $ (1,3) $ and $ (1,4) $ via column operations.
We get
\[
	\det(A_4' ) 
	= 
	\det
	\begin{pmatrix}
	0  &  1 & 0 & 0
	\\
	-1  &  0 &  x_{2,3}' & x_{2,4}'
	\\
	-x_{1,3}'  &  0 & x_{1,3}'x_{2,3}' & x_{3,4}' + x_{1,4}' x_{2,3}'
	\\ 
	-x_{1,4}'  &  0 & -x_{3,4}' + x_{1,3}' x_{2,4}'    & x_{1,4}'x_{2,4}'
	\end{pmatrix}
	.
\]

Next, we add $ x'_{2,3} $-times the first column on the third column and $ x'_{2,4} $-times the first column on the fourth column 
	to obtain that the second row becomes $ (-1,0,0,0) $.
	Then, we eliminate the entries at position $ (3,1) $ and $ (4,1) $ via row operations.
	This provides 
\[
	\det(A_4')= 
	\det 
	\begin{pmatrix}
	0  &  1 & 0 & 0
	\\
	-1  &  0 &  0 & 0
	\\
	0  &  0 & 0 & x_{3,4}' + x_{1,4}' x_{2,3}' - x_{1,3}' x_{2,4}'
	\\ 
	0  &  0 & -x_{3,4}' + x_{1,3}' x_{2,4}'-x_{1,4}' x_{2,3}'    & 0
	\end{pmatrix}
	.
\] 
By expanding the determinant on the right hand side, we obtain
\[
\det(A_4')
=
	\det 
	\begin{pmatrix}
	 0 & y_{3,4}
	\\ 
	-y_{3,4} & 0
	\end{pmatrix}
	= y_{3,4}^2,
	\ \ \ \mbox{for } \ 
	y_{3,4} := x_{3,4}' + x_{1,4}' x_{2,3}' - x_{1,3}' x_{2,4}'. 
\]

Therefore, $ ( Y_{4,4}^{\rm skew} )'_{\rm red} $  coincides with the
divisor $ \operatorname{div}(y_{3,4}) $ in the given chart.
The latter is regular and transversal to the exceptional divisor $ \operatorname{div}(x_{1,2}) $.
Therefore, the singularities of the reduction are resolved. 

We point out that $ y_{3,4} $ is the strict transform of the $ 3 $-minor of $ A_4 $ containing $ x_{1,2} $ twice and $ x_{3,4} $, i.e., the determinant of $(A_4)_{\{1,2,3\},\{1,2,4\}}$:
\[
	x_{1,2} \big( 
	x_{1,2} x_{3,4} + x_{1,4} x_{2,3} - x_{1,3} x_{2,4}
	\big) 
	=   
	\det \begin{pmatrix}
	0 & x_{1,2} & x_{1,4} \\
	-x_{1,2} & 0 & x_{2,4} \\
	-x_{1,3} & -x_{2,3} & x_{3,4}
	\end{pmatrix}.
\]
Furthermore, note that for every non-zero $ 3 $-minor the corresponding matrix contains exactly one element $ x_{i,j} $ twice
and it is up to sign of the form
\[
	x_{i,j} \big( 
	x_{1,2} x_{3,4} + x_{1,4} x_{2,3} - x_{1,3} x_{2,4}
	\big).
\]
This illustrates \eqref{eq:2l_2l-1} in the given explicit situation for $ \ell = 2$.
\end{Ex} 

For completeness, let us also have a look at the skew-symmetric case for $ m = 3 $.

\begin{Ex}[Theorem~\ref{Thm_1} for $ m = 3 $]
	\label{Ex:33skew}
	As explained in the introduction, 
		$ Y_{3,1}^{\rm skew} $ 
	and $ Y_{3,3}^{\rm skew} $ are regular.
	Hence, by \eqref{eq:2l_2l-1}, the same is true for the reduction of $ Y_{3,2}^{\rm skew} $.
\end{Ex}

\section{Proof of Theorem \ref{Thm_1}}
\label{sec:proof1}

We come to the general proof of Theorem~\ref{Thm_1}.
First, we recall the setup and fix some abbreviations:
\begin{itemize}
	\item 
	$ R_0 $ is a regular ring of characteristic different from $ 2 $.
	
	\item 
	$ S := S_{\rm skew} \cong R_0 [x_{i,j} \mid 1 \leq i < j \leq m ] $
	and $ Z := Z_{\rm skew} = \Spec(S) $. 
	
	\item 
	$ A_{m} =  
	\begin{pmatrix}
		0  &  x_{1,2} & \cdots     & x_{1,m}
		\\
		-x_{1,2}  &  0 & \cdots     & x_{2,m}
		\\
		\vdots & \vdots & \ddots  & \vdots    
		\\ 
		-x_{1,m}  &  -x_{2,m} & \cdots     & 0
	\end{pmatrix} $
	and
	$ Y_{m,r} := Y_{m,r}^{\rm skew} = \Spec (S/\cI_{r} (A_m)) $.
\end{itemize}
By~\eqref{eq:2l_2l-1}, we may restrict to the case $ r = 2 \ell $. 
In Theorem~\ref{Thm_1}, we proposed an explicit desingularization for the reduction of $ Y_{m,2\ell}^{\rm skew} $. 

\begin{proof}[Proof of Theorem~\ref{Thm_1}]
	First, observe that the reduction of $ Y_{m,2} $ is always regular, as $ Y_{m,1} $ is isomorphic to $ \Spec(R_0) $ and $ R_0 $ is a regular ring by assumption.	
	
	We consider the case of $ Y_{m,m} $. 
		As the examples of the previous section indicate, we will prove the statement for $ Y_{m,r} $ with $ r < m $ along the way.
	We perform an induction on the size $ m \in \ZZ_+ $ of the matrix $ A_m $.
	Notice that we include the case $ m $ odd here 
	(even though we already know that $ Y^{\rm skew}_{2\ell+1,2\ell+1} $ is regular)
	in order to provide a unified presentation.
	For $ m = 1 $, there is nothing to prove. 
	Furthermore, we treated the cases $ m \in \{ 2, 3, 4 \} $ in Examples~\ref{Ex:22}, \ref{Ex:33skew}, \ref{Ex:44skew}, respectively.
	Therefore, we go on to the induction step and assume $ m \geq 5 $ in the following.

	The first center proposed by Theorem~\ref{Thm_1} is $ D_1 := Y_{m,1} \cong  (Y_{m,2})_{\rm red} $,
	 which is regular.
	 Let us look at a chart of the blowing up with center $ D_1 $. 
	 Without loss of generality, we consider the 
	 $ X_{1,2}$-chart.
 	(Note that for the $ X_{k,\ell} $-chart, we may perform a suitable interchange of columns and rows followed by a renaming of the variables in order to attain the same setting as in the $ X_{1,2} $-chart, up to a possible interchange of signs.)
 
 	In the given chart, the strict transform $ Y_{m,m}' $ of $ Y_{m,m} $ is determined by the strict transform $ f' $ of $ f := \det (A_m) $.
	We have that $f'$ is equal to the following determinant
	\[
		f'=
		\det  
			\begin{pmatrix}
				0  &  1 &    x'_{1,3}& \cdots  & x'_{1,m}
				\\
				-1  &  0 &   x'_{2,3}  & \cdots & x'_{2,m}
				\\
				-x'_{1,3}  &  -x'_{2,3} & 0&\cdots     & x'_{3,m}
				\\
				\vdots & \vdots & & \ddots  & \vdots    
				\\ 
				-x'_{1,m}  &  -x'_{2,m} & -x'_{3,m}& \cdots     & 0
			\end{pmatrix} .
	\]
	We eliminate all entries in the second column of the matrix except for the $ 1 $ at position $ (1,2) $ by adding the first row $ x_{2,i}' $-times to the $ i $-th row, for $ i \geq 3 $.
	After that we perform column operations to eliminate all entries in the first row except for the $ 1 $ at position $(1,2) $. 
	The latter step has no effect on the other entries of the matrix since we cleaned the second column beforehand.
	We obtain that $f'$ is equal to 
	\[
		\det 
		\begin{pmatrix}
			
			0  &  1 &    0& \cdots  & 0
			\\
			-1  &  0 &   x'_{2,3}  & \cdots & x'_{2,m}
			\\
			-x'_{1,3}  &  0 & x'_{1,3}x'_{2,3} &\cdots     & x'_{3,m}+x'_{1,m}x'_{2,3}
			\\
			\vdots & \vdots & & \ddots  & \vdots    
			\\ 
			-x'_{1,m}  &  0 & -x'_{3,m}+x'_{1,3}x'_{2,m}& \cdots     & x'_{1,m}x'_{2,m}
		\end{pmatrix}
		.
	\]
	In order to regain the skew-symmetry, 
	we first add  $(-x_{1,i})$-times the second row to the $ i $-th row, for $ i \geq 3 $,
	which eliminates all entries in the first column except for the $ -1 $ at position $ (2,1) $.
	Using the latter entry, we then perform column operations to clean up the second row.
	This
	leads to 		
	\[
		f'=\det  
			\begin{pmatrix}
				 0 & 1 & 0 &0& \cdots & 0\\
				 -1&  0 & 0 &0& \cdots & 0\\
				   0&0&0&y_{3,4} &\cdots  & y_{3,m}
				\\
				   0&0&-y_{3,4}  &0& \cdots & y_{4,m}
				\\
				\vdots & \vdots & \vdots & \vdots & \ddots  & \vdots    
				\\ 
				0&0& -y_{3,m}  &  -y_{4,m} & \cdots     & 0
			\end{pmatrix},
	\]
	where we introduce
	\begin{equation}
		\label{eq:y_ij} 
		y_{i,j} := x'_{i,j}-x'_{2,j}x'_{1,i}+x'_{1,j}x'_{2,i},
		\ \ \   
		\mbox{ for } 
		\
		3 \leq i < j \leq m .
	\end{equation}
	
	Observe that $(y_{i,j})$ are independent variables 
	and they are transversal to the variable $ x'_{2,1} $,
	which describes the exceptional divisor of the blowing up in the present chart.

	Finally, we expand the determinant with respect to the first two columns and get
	\[
		f'=\det 
		\begin{pmatrix}
			   0&y_{3,4} &\cdots  & y_{3,m}
			\\
			   -y_{3,4}  &0& \cdots & y_{4,m}
			\\
			\vdots & \vdots & \ddots  & \vdots    
			\\ 
			-y_{3,m}  &  -y_{4,m} & \cdots     & 0
		\end{pmatrix}.
	\]
	
	\medskip 
		
	Note that the latter is the determinant of the generic skew-symmetric matrix of size $ m-2 $.
	Therefore, by induction, Theorem~\ref{Thm_1} provides an embedded resolution for the reduced subscheme determined by the vanishing of $ f' $.
		
	It remains to show that the desingularization obtained by induction coincides with the one proposed in the statement of Theorem~\ref{Thm_1}.
	In other words, we have to prove that the strict transform $ Y_{m,r}' $ of $ Y_{m,r} $ is equal to 
	$ Y_{m-2,r-2} $ (up to renaming the generic variables) in the present chart,
	i.e., if we set 
	\[
		A_m' := 
		\begin{pmatrix}
		0  &  1 &    x'_{1,3}& \cdots  & x'_{1,m}
		\\
		-1  &  0 &   x'_{2,3}  & \cdots & x'_{2,m}
		\\
		-x'_{1,3}  &  -x'_{2,3} & 0&\cdots     & x'_{3,m}
		\\
		\vdots & \vdots & & \ddots  & \vdots    
		\\ 
		-x'_{1,m}  &  -x'_{2,m} & -x'_{3,m}& \cdots     & 0
		\end{pmatrix},
		\ \ \
		\widetilde{A}_{m-2} := 
		\begin{pmatrix}
		0&y_{3,4} &\cdots  & y_{3,m}
		\\
		-y_{3,4}  &0& \cdots & y_{4,m}
		\\
		\vdots & \vdots & \ddots  & \vdots    
		\\ 
		-y_{3,m}  &  -y_{4,m} & \cdots     & 0
		\end{pmatrix}
	\]
	and use the notation introduced in Definition~\ref{Def:IrM},
	then we have to prove:
	\begin{equation}
	\label{eq:to_show_Am} 
		\cI_{r-2} (\widetilde{A}_{m-2}) = \cI_r (A_m'). 
	\end{equation}
	This will also settle the case for $ m $ odd and $ r < m $.
	
	Notice that \eqref{eq:to_show_Am} provides that the desingularizations obtained by induction in every chart of the first blowing up glue to a global resolution of singularities as desired
	since the strict transform of the original minors are a global center.
	Moreover, we obtain from this that $ (Y_{m,3}')_{\rm red}  =  (Y_{m,4}')_{\rm red}  $ 
	is regular and transversal to the exceptional divisor $ \operatorname{div}(x_{1,2}') $
	since the last two properties are true for $ Y_{m-2,1} $. 
	
	\smallskip 

	We already explained that $ Y_{m,m}' $ identifies with $ Y_{m-2,m-2} $.
	Moreover, we have $ Y_{m,1}' = \varnothing $ (since the strict transform of the center is always empty)
	and $ Y_{m,2}' = \varnothing $ 
	(since we have $ (Y_{m,2})_{\rm red} \cong Y_{m,1} $).

	Let $ r \geq 3 $.
	By Example~\ref{Ex:std} and Lemma~\ref{Lem:Trafo}, 
	$ Y_{m,r}' $ is generated by the $ r $-minors of the matrix $ A_m' $.
	As we have seen before, after some elementary row and column operation the matrix $ A_m' $ becomes
	\[
		\begin{pmatrix}
		0 & 1 & 0 \\
		-1&  0 & 0 \\
		0&0&  \widetilde{A}_{m-2}
		\end{pmatrix}.
	\]
	
	For every $ r $-minor $ g_{I,J} $ of $ A_m' $ 
	corresponding to $ I, J \subset\{ 1, \ldots, m \} $ with $ J \supset \{ 1,2 \} \subset I $,
	we may perform the same row and column operations and obtain that it is equal to the $ (r-2) $-minor of the matrix 
	$ 
		\widetilde{A}_{m-2}
	$,  
	whose index sets correspond to $ I \setminus \{ 1,2 \} $ and $ J \setminus \{ 1, 2 \} $ (up to a shift in the index set).
	This shows
	\[
		\cI_{r-2} (\widetilde{A}_{m-2}) \subseteq \cI_r (A_m'). 
	\]
	
	We are left with the task to show that 
	the inclusion is an equality, 
	$ \cI_{r-2} (\widetilde{A}_{m-2}) = \cI_r (A_m') $.
	For this, we have to prove that any $ r $-minor of $ A_m ' $ which we did not consider for the first inclusion is contained in $ \cI_{r-2} (\widetilde{A}_{m-2}) $.
	The latter are those $ r $-minors $ g_{I,J} $ of $ A_m' $  with $ \{1,2\} \not\subset I $ or $ \{1,2\} \not\subset J $.

	We introduce some notations. 
	Let $ I =: \{ i(1), \ldots, i(r) \} $ with $ i(1) < i(2) < \ldots < i(r) $. 
	For $ j \in J $, we define 
	\[
		\kappa := \kappa(j) := 
		\begin{cases}
			\max \{ \alpha \in \{ 1, \ldots, r \} \mid i(\alpha) \leq j \} ,
			& \mbox{if } i(1) \leq j ,
			\\ 
			0, & \mbox{if } i(1) > j.	
		\end{cases}
		.
	\]
	If $ \kappa(j) = j $, we set $ x_{j,j}' := 0 $ and $ y_{j,j} := 0 $.
	Denote by $ \cM_{I,J} $ the $ r \times r $ submatrix of $ A_m $ whose rows are indexed by $ I $ and whose columns are index by $ J $.
	Hence, $ g_{I,J} = \det (\cM_{I,J} ) $.

	First, we discuss the case that $ \{1,2\} \not\subset I $ and $ \{1,2\} \not\subset J $.
	Using the definition of $ y_{i,j} $ \eqref{eq:y_ij} to substitute $ x_{i,j}' $ in $ \cM_{I,J} $,
	we get
	\[
	\cM_{I,J} 
	=   
	\begin{pmatrix}
		x_{i(1),j}'
		\\
		\vdots
		\\ 
		x_{i(\kappa),j}'
		\\
		- x_{j,i(\kappa+1)}'
		\\
		\vdots 
		\\
		- x_{j,i(r)}'	
		\end{pmatrix}_{j \in J}
	=	
		\begin{pmatrix}
			y_{i(1),j} + x'_{2,j} x'_{1,i(1)}  - x'_{1,j} x'_{2,i(1)}
			\\
			\vdots
			\\ 
			y_{i(\kappa),j} + x'_{2,j} x'_{1,i(\kappa)}  - x'_{1,j} x'_{2,i(\kappa)}
			\\
			-y_{j,i(\kappa+1)} - x'_{2,i(\kappa+1)} x'_{1,j}  + x'_{1,i(\kappa+1)} x'_{2,j}
			\\
			\vdots 
			\\
			- y_{j,i(r)} - x'_{2,i(r)} x'_{1,j} + x'_{1,i(r)} x'_{2,j}
		\end{pmatrix}_{j \in J}
	.
	\]
	Notice that, if $ i(\kappa) = j $, then the entry at position $ (i(k),j) $ is zero. 
	Further, if $ \kappa(j) = 0 $, then $ j $-th column begins with $ - x_{j,i(1)}' $. 
	
	Set $ E_{2} := \begin{pmatrix}
		1 & 0 \\ 0 & 1 
	\end{pmatrix} $ and 
	$ \mathcal{B} :=  \begin{pmatrix}
	E_{2} & 0
	\\
	\cA & \cM_{I,J} 
	\end{pmatrix}  $, 
	where $ \cA $ is an arbitrary $ r \times 2 $ matrix and $ 0 $ stands for the matrix of size $ 2 \times r $ with all entries zero.
 	We have 
	\[
		\det ( \mathcal{B}  ) = \det (\cM_{I,J}) = g_{I,J} . 
	\]
	Let us consider the matrix $ \cA $ given by
	\[
		\cA := 
		\begin{pmatrix}
			- x'_{1,i(1)}  &  -x'_{2,i(1)}
			\\
			\vdots & \vdots  
			\\
			 -x'_{1,i(r)}  &  -x'_{2,i(r)}
		\end{pmatrix}	
	.	
	\]
	Observe that the entries of $ \cA $ are coming from the first two columns of $ A_m' $.

	In the next step, we use the entries of $ \mathcal{A} $ 
	to eliminate the terms $ x'_{2,j} x'_{1,i}  - x'_{1,j}x'_{2,i} $ in the columns of $ \mathcal{B} $ corresponding to $ \cM_{I,J} $ (whose index ranges in $ J $).
	More precisely, we apply the following elementary column operations to $ \mathcal B $ for every $ j \in J $:
	\begin{itemize}
		\item 
		Add $ x_{2,j}' $-times the first column of $ \mathcal{B} $ to the column with index $ j $;
		
		\item 
		add $ (-x_{1,j}') $-times the second column of $ \mathcal{B} $ to the column with index $ j $.
	\end{itemize}
	This leads to 
	
	\[
	g_{I,J} 
	= 
\det \left( 
	\begin{array}{cc}
		1 &  0 
		\\
		0  & 1 
		\\
		- x'_{1,i(1)}  &  -x'_{2,i(1)}
		\\
		\vdots & \vdots 
		\\ 
		- x'_{1,i(\kappa)}  &  -x'_{2,i(\kappa)}
		\\
		- x'_{1,i(\kappa+1)}  &  -x'_{2,i(\kappa+1)}
		\\
		\vdots & \vdots 
		\\
		-x'_{1,i(r)}  &  -x'_{2,i(r)}
	\end{array} 
	\begin{pmatrix}
		x_{2,j}' 
		\\
		-x_{1,j}' 
		\\
		y_{i(1),j} 
		\\
		
		\vdots
		\\ 
		y_{i(\kappa),j}
		\\
		-y_{j,i(\kappa+1)} 
		\\
		\vdots 
		\\
		- y_{j,i(r)}
	\end{pmatrix}_{j \in J}
\right) 
.
	\]
	
	\noindent
	We expand the determinant with respect to the second column so that it becomes a sum of determinants of matrices of size $ ( r+1 ) \times (r+1) $.
	(Of course, there are certain coefficients appearing in the sum, but these will not be relevant, so we neglect them here).
	We continue and expand each of these matrices with respect to their respective first column so that $ g_{I,J} $ is expressed as a sum of determinants of $ r \times r $ matrices. 
	Observe that the latter matrices fulfill the property that if we neglect the first and second row then all entries are of the form $ y_{i,j} $ or $ -y_{i,j} $ for some $ i \in I $ and $ j \in J $. 
	We proceed and expand the determinants of the appearing $ r \times r $ matrices with respect to the second row
	and after that we expand the resulting determinants of $ (r-1) \times (r-1) $ matrices with respect to the first row.
	In total, due to the described property of the $ r \times r $ matrices, 
	we have written $ g_{I,J} $ as a sum of $ (r-2) $-minors of the matrix $ \widetilde{A}_{m-2} $,
	i.e., $ g_{I,J} \in \cI_{r-2} (\widetilde{A}_{m-2}) $,
	as claimed.

	It is not hard to adapt the argument for the remaining
	$ r $-minors $ g_{I,J} $ of $ A_m' $ with $ \{1,2\} \not\subset I $ or $ \{1,2\} \not\subset J $:
	First, one has to enlarge the matrix $ \cM_{I,J} $ determining $ g_{I,J} $ 
	by the needed entries of the first and second column of $ A_m' $ which are not already appearing in $ \cM_{I,J} $.
	Note that if $ \{1,2\} \subset J $, then no enlargement is required.
	On the other hand, observe, if $ 1 \in I $ (resp.~$ 2 \in I $), then the corresponding row for $ \cA $ is induced by $ (0 \  1 ) $ (resp.~$ (-1 \ 0 ) $).
	After that, we substitute $ x_{i,j}' $ using \eqref{eq:y_ij}, for $ 3 \leq i < j \leq m $, and apply elementary column operations in order to reach that these entries are of the form $ \pm y_{i,j} $.
	If $ 1 \in I $ (or $ 2 \in I $), then these column operations eliminate all entries of the first (resp.~second) row of the $ \cM_{I,J} $-part of the analogue of $ \mathcal B $.
	Finally, by a suitable expansion of the determinant, one achieves the desired statement
	$ g_{I,J} \in \cI_{r-2} (\widetilde{A}_{m-2}) $. 
	
	In conclusion, we proved \eqref{eq:to_show_Am}, which ends the proof.
\end{proof}

\section{Proof of Theorem~\ref{Thm_2}}
\label{sec:proof2}

Let us turn to the generic symmetric case. 
We fix the notation:

\begin{itemize}
	\item 
	$ R_0 $ is a regular ring.
	
	\item 
	$ S := S_{\rm sym} \cong R_0 [x_{i,j} \mid 1 \leq i \leq j \leq m ] $
	and $ Z := Z_{\rm sym} = \Spec(S) $. 
	
	\item 
	$ B_{m} =  
	\begin{pmatrix}
		x_{1,1}  &  x_{1,2} & \cdots     & x_{1,m}
		\\
		x_{1,2}  &  x_{2,2} & \cdots     & x_{2,m}
		\\
		\vdots & \vdots & \ddots  & \vdots    
		\\ 
		x_{1,m}  &  x_{2,m} & \cdots     & x_{m,m}
	\end{pmatrix} $
	and
	$ Y_{m,r} := Y_{m,r}^{\rm sym} = \Spec (S/\cI_{r} (B_m)) $.
\end{itemize}

\noindent 
The goal of this section is to show that the sequence of blowing ups described in Theorem~\ref{Thm_2} is an embedded resolution of singularities for $ Y_{m,r}^{\rm sym} $.

\begin{proof}[Proof of Theorem~\ref{Thm_2}]
As in the proof of Theorem~\ref{Thm_1}, we perform an induction on the size $m\in \mathbb Z_+$ of the matrix $B_m$. 
If $ m = 1 $, there is nothing to prove. 
In Examples~\ref{Ex:22} and~\ref{Ex:33sym}, we treated the cases $ m \in \{ 2, 3 \} $.
Thus, let $ m \geq 4 $. 

Analogous to before, we study $ Y_{m,m} $ and handle the cases $Y_{m,r}$ with $ r < m $ along the way.  
Clearly, $Y_{m,1}$ is always regular, since $Y_{m,1}\cong \Spec(R_0)$ and $R_0$ is assumed to be regular. 
Following Theorem~\ref{Thm_2}, we blow up the center $ D_1 := Y_{m,1} $.

\smallskip 
 
We begin with the $X_{1,1}$-chart as representative of the $X_{i,i}$-charts for $  i \in \{ 1, \ldots, m \} $
(since by suitably interchanging columns and rows followed by renaming variables, we reach the $ X_{1,1} $-chart for the other  $X_{i,i}$-charts). 
The strict transform $f'$ of $ f := \det ( B_m) $ is given by the determinant of the matrix
\[
B_m' :=  
\begin{pmatrix}
	1  &  x'_{1,2} & \cdots     & x'_{1,m}
	\\
	x'_{1,2}  &  x'_{2,2} & \cdots     & x'_{2,m}
	\\
	\vdots & \vdots & \ddots  & \vdots    
	\\ 
	x'_{1,m}  &  x'_{2,m} & \cdots     & x'_{m,m}
\end{pmatrix}. 
\]
We eliminate the entries in the first column by subtracting $ x_{1,j}$-times the first row from the $j$-th row for $j \in \{ 2, \ldots, m \} $.
We get
\[
f'=\det  
\begin{pmatrix}
	1  &  x'_{1,2} & \cdots     & x'_{1,m}
	\\
	0 &  x'_{2,2}-x'^2_{1,2} & \cdots     & x'_{2,m}-x'_{1,m}x'_{1,2}
	\\
	\vdots & \vdots & \ddots  & \vdots    
	\\ 
	0  &  x'_{2,m}-x'_{1,2}x'_{1,m} & \cdots     & x'_{m,m}-x'^2_{1,m}
\end{pmatrix}
=\det  
\begin{pmatrix}
	y_{2,2} & \cdots     & y_{2,m}
	\\
	\vdots & \ddots  & \vdots    
	\\ 
	y_{2,m} & \cdots     &y_{m,m}
\end{pmatrix},
\]
where we expand the determinant with respect to the first column in the last equality and introduce 
\[
y_{i,j} := x'_{i,j}-x'_{1,i}x_{1,j}' 
\hspace{1cm} \mbox{ for } i, j \in \{ 2, \ldots, m \}.
\]
These $ (y_{i,j})_{i,j} $ are independent variables which are transversal to the exceptional variable $x'_{1,1}$.
This shows that $f'$ is defined by the determinant of the generic symmetric matrix of size $ m-1 $.
In other words, in the given chart, 
the strict transform $ Y_{m,m}' $ of $ Y_{m,m} $ is isomorphic to $ Y_{m-1,m-1} $. 
By induction, the remaining resolution process is determined. 

Let $ \widetilde{B}_{m-1} $ be the symmetric matrix with entries $ y_{i,j} $.
Analogous to the proof of Theorem~\ref{Thm_1}, it remains to prove that the ideal generated by the $( r-1 ) $-minors of $ \widetilde{B}_{m-1} $ coincides with the strict transform of the $ r $-minors of $ B_m $, 
\[
	\cI_{r-1} (\widetilde{B}_{m-1}) = \cI_r (B_m'), 
\] 
for $ r \geq 2 $.
We already proved the equality for $ r = m $.  
Hence, let $ 1 < r < m $. 
The inclusion $ ``\subseteq" $ is clear since the $ (r-1) $-minors of $ \widetilde{B}_{m-1} $ correspond to the $ r $-minors of $ B_m' $ involving  the entry of $ B_m' $ at position $ (1,1) $. 
For the reverse inclusion $ ``\supseteq" $, we apply the analogous arguments as in the proof of Theorem~\ref{Thm_1}: 
We consider an $ r $-minor of $ B_m' $, apply the definition of $ y_{i,j} $ to substitute the corresponding $ x_{i,j}' $, if necessary we enlarge the matrix by the respective entries of the first column of $ B_m' $
to perform suitably column operations and expand the determinant to reach the desired inclusion.

\smallskip

 Next let us have a look at the $X_{1,2}$-chart, which is a representative of the $X_{i,j}$-charts for $1\le i<j \le m $. 
The strict transform $f'$ of $f = \det(B_m)$ is equal to
 \[
	f' = \det (B_m'),
	\hspace{0.7cm} 
	\mbox{ for } 
	B_m' :=  
		\begin{pmatrix}
			x'_{1,1}  &  1 & \cdots     & x'_{1,m}
			\\
			1  &  x'_{2,2} & \cdots     & x'_{2,m}
			\\
			\vdots & \vdots & \ddots  & \vdots    
			\\ 
			x'_{1,m}  &  x'_{2,m} & \cdots     & x'_{m,m}
		\end{pmatrix}. \]
		
		We eliminate all entries in the first row and second column by subtracting $ x'_{2,j} $-times the first row from the $j$-th row for $j \in \{  2,\ldots,m \}$
		and then cleaning up the first row analogous to before.
		We obtain 
		
		\[
		f'= \det  
		\begin{pmatrix}
			0  &  1 & 0& \cdots     & 0
			\\
			1-x'_{1,1}x'_{2,2} &  0 & x'_{2,3}-x'_{2,2}x'_{1,3}& \cdots     & x'_{2,m}-x'_{2,2}x'_{1,m}
			
			\\
						x'_{1,3}-x'_{1,1}x'_{2,3} & 0 & x'_{3,3}-x'_{2,3}x'_{1,3} &\cdots & x'_{3,m}-x'_{2,3}x'_{1,m}  
			\\
			\vdots & \vdots & \vdots &\ddots  & \vdots    
			\\ 
			x'_{1,m}-x'_{1,1}x'_{2,m}  &  0 & x'_{3,m}-x'_{2,m}x'_{1,3}& \cdots     & x'_{m,m}-x'_{2,m}x'_{1,m}
		\end{pmatrix}.
		\]
	
		Since we already discussed the $ X_{i,i} $-charts, we may neglect the part of the given $ X_{1,2} $-chart which is also contained in some $ X_{i,i} $-chart, for $ i \in \{ 1, \ldots, m \} $. 
		In particular, we may assume that $ x_{1,1}' = \frac{X_{1,1}}{X_{1,2}} $ is  not a unit 
		(using the notation of Remark~\ref{Rk:blowup}),
		cf.~Example~\ref{Ex:33sym}.
		Hence, may assume that 
		$\varepsilon := 1-x'_{1,1}x'_{2,2}$ is invertible in present chart. 
		 
		In order to regain the symmetry, we multiply the first column by $\varepsilon^{-1}$
		and multiply the $j$-th row by $\varepsilon$, for $ j \in \{ 3, \ldots, m \} $. 
		This provides
		
		\[
		f'= \phi \cdot 
		\det\begin{pmatrix}
			 0& 1 & 0 &\cdots & 0\\  
			1 & 0 & x'_{2,3}-x'_{2,2}x'_{1,3} & \cdots & x'_{2,m}-x'_{2,2}x'_{1,m}\\
		x'_{1,3}-x'_{1,1}x'_{2,3} & 0 &\varepsilon ( x'_{3,3}-x'_{2,3}x'_{1,3}) & \cdots & \varepsilon(x'_{3,m}-x'_{2,3}x'_{1,m})	
			
			\\
			\vdots & \vdots & \vdots &\ddots  & \vdots    
			\\ 
			x'_{1,m}-x'_{1,1}x'_{2,m} & 0 & \varepsilon(x'_{3,m}-x'_{2,m}x'_{1,3})& \cdots     & \varepsilon(x'_{m,m}-x'_{2,m}x'_{1,m})
		\end{pmatrix}, \]
		for $ \phi $ invertible. 
		Hence, $ \phi $ may be neglected when considering the vanishing locus. 
		
		In the next step, we use the $ 1 $ at position $ (2,1) $ to eliminate all other entries of the first column.
		More precisely, for $ i \in \{ 3, \ldots , m \} $,
		we subtract
		$ (x'_{1,i}-x'_{1,1}x'_{2,i}) $-times the second row
		 from the $ i $-th row.
		 Afterwards, we expand the determinant with respect to the first two columns.  
		Altogether, this leads to 
		
		\[
		 f'=
		(-1) \cdot \phi \cdot  \det ( \widetilde{B}_{m-2}) 
	, 
	\]
		where 
		\[
			\widetilde{B}_{m-2} := \begin{pmatrix}
				y_{3,3} &\cdots & y_{3,m}\\
				\vdots &\ddots  & \vdots    
				\\ 
				y_{3,m}& \cdots     & y_{m,m}
			\end{pmatrix}
		\]
		and
		\[ 
			y_{i,j} := 
			\varepsilon (x'_{i,j}-x'_{2,i} x'_{1,j}) - (x'_{1,i}-x'_{1,1}x'_{2,i})(x'_{2,j}-x'_{2,2}x'_{1,j}),
		\] 
		for $ i, j \in \{ 3, \ldots, m \} $.
		Note that this definition is well-defined 
		since $\varepsilon$ is a unit
		and $ (y_{i,j})_{i,j} $ are independent variables which are transversal to the exceptional variable $x'_{1,2}$.
		
		In conclusion, the strict transform $ Y_{m,m}' $ of $ Y_{m,m} $ in the present chart is isomorphic to $ Y_{m-2,m-2} $.
		By induction, Theorem~\ref{Thm_2} holds for $ Y_{m-2,m-2} $ and the remaining embedded resolution process is known. 
		Analogous to before, it remains to prove 
		$ \cI_{r-2} (\widetilde{B}_{m-2}) = \cI_r (B_m') $,
		for $ r \in \{ 3, \ldots, m-1 \} $.
		Since similar arguments as before apply,
		we leave the details to the reader as an exercise. 
		
		Notice that in the present chart, 
		the vanishing locus of the $2$-minors of $ B_m' $ is empty, since we assume that $\varepsilon$ is invertible. 
		Therefore, from a global perspective, 
		we do not see the second center of the resolution procedure described in the statement of Theorem~\ref{Thm_2} here.

	\smallskip 
	
	In conclusion, we deduced that the local computations in each chart coincide with the global sequence of blowing ups described in Theorem~\ref{Thm_2} . 
	Furthermore, an inductive argument shows that this is indeed an embedded resolution of singularities.
\end{proof} 

\begin{acknowledgements}
	Bernd Schober thanks Laura Escobar for joint discussions on (skew-) symmetry preserving decompositions of matrices in the context of a different topic which inspired the methods for reduction in the present article. \\
	Both authors thank the referees for useful comments on an earlier version of the article and in particular, for pointing out parts that needed clarification in the proof of Theorem~\ref{Thm_1}. 
\end{acknowledgements}

\noindent 
{\bf Declarations}

\smallskip 

\noindent 
{\bf Data Availibility Statement } 
Data sharing not applicable to this article as no data sets were
generated or analyzed during the current study.

\smallskip 

\noindent 
{\bf Conflict of interest } On behalf of all authors, the corresponding author states that there is no conflict of interest.

\def\cprime{$'$}


\begin{thebibliography}{BGH22}
	
	\bibitem[BGH]{BGH}
	J.~W. Bruce, V.~V. Goryunov, and G.~J. Haslinger.
	\newblock Families of skew-symmetric matrices of even size, 2022.
	\newblock{Preprint available on \url{https://arxiv.org/abs/2206.00596}}.
	
	\bibitem[BM]{BM}
	Edward Bierstone
	and Pierre D. Milman. 
	\newblock {Canonical desingularization in characteristic zero by blowing up the maximum strata of a local invariant}. 
	\newblock {\em Invent.~Math.} 128 (1997), no.~2, 207--302.
	
	\bibitem[BV]{BV}
	Winfried Bruns and Udo Vetter.
	\newblock {\em Determinantal rings}, volume 1327 of {\em Lecture Notes in
		Mathematics}.
	\newblock Springer-Verlag, Berlin, 1988.
	
	\bibitem[Cut]{Cu}
	Steven~Dale Cutkosky.
	\newblock {\em Resolution of singularities}, volume~63 of {\em Graduate Studies
		in Mathematics}.
	\newblock American Mathematical Society, Providence, RI, 2004.
	
	\bibitem[EV]{EV}
	Santiago Encinas and Orlando Villamayor. 
	\newblock A course on constructive desingularization and equivariance. 
	\newblock {\em Resolution of singularities (Obergurgl, 1997)}, 147--227, Progr.~Math., 181, {\em Birkhäuser, Basel}, 2000.
	
	\bibitem[FKZ]{FKZ}
	Anne Frühbis-Krüger and Matthias Zach.
	\newblock Determinantal singularities.
	\newblock{\em Handbook of geometry and topology of singularities IV}, 45--159, {\em Springer, Cham}, 2023.
	
	\bibitem[Gau1]{SabrinaThesis}
	Sabrina~Alexandra Gaube.
	\newblock {\em Specialized strategies for resolution of singularities of
		determinantal ideals}.
	\newblock PhD thesis, Carl von Ossietzky Universität Oldenburg, 2023.
	
	\bibitem[Gau2]{Gaube}
	Sabrina~Alexandra Gaube.
	\newblock Resolution of determinantal ideals with exploition of the matrix
	structure.
	\newblock{\em In preparation}.
	
	\bibitem[GM1]{Gaffney1}
	Terence Gaffney and Michelle Molino.
	\newblock Symmetric determinantal singularities I: The multiplicity of the
	polar curve, 2020.
	
	\bibitem[GM2]{Gaffney2}
	Terence Gaffney and Michelle Molino.
	\newblock Symmetric determinantal singularities II: Equisingularity and seids,
	2021.
	
	\bibitem[Ha]{Harris}
	Joe Harris.
	\newblock {\em Algebraic geometry}, volume 133 of {\em Graduate Texts in
		Mathematics}.
	\newblock Springer-Verlag, New York, 1992.
	\newblock A first course.
	
	\bibitem[Hi]{Hiro64}
	Heisuke Hironaka.
	\newblock Resolution of singularities of an algebraic variety over a field of
	characteristic zero. {I}, {II}.
	\newblock {\em Ann. of Math. (2)}, 79:109--203; 79 (1964), 205--326, 1964.
	
	\bibitem[KLS]{KLS}
	Vijay Kodiyalam, T.~Y. Lam, and R.~G. Swan.
	\newblock Determinantal ideals, {P}faffian ideals, and the principal minor
	theorem.
	\newblock In {\em Noncommutative rings, group rings, diagram algebras and their
		applications}, volume 456 of {\em Contemp. Math.}, pages 35--60. Amer. Math.
	Soc., Providence, RI, 2008.
	
	\bibitem[Kol]{Kollar}
	J\'{a}nos Koll\'{a}r.
	\newblock {\em Lectures on resolution of singularities}, volume 166 of {\em
		Annals of Mathematics Studies}.
	\newblock Princeton University Press, Princeton, NJ, 2007.
	
	\bibitem[Sch]{BerndPartial}
	Bernd Schober.
	\newblock Partial local resolution by characteristic zero methods.
	\newblock {\em Results Math.}, 73(1):Paper No. 48, 39, 2018.
	
	\bibitem[Vai]{DetermResol}
	Israel Vainsencher.
	\newblock Complete collineations and blowing up determinantal ideals.
	\newblock {\em Math. Ann.}, 267(3):417--432, 1984.
	
	\bibitem[V1]{Vil1}
	Orlando Villamayor. 
	\newblock {Constructiveness of Hironaka's resolution}. 
	\newblock {\em Ann.~Sci.~École Norm.~Sup. (4)} 22 (1989), no.~1, 1--32.
	
	\bibitem[V2]{Vil2} 
	Orlando~E. Villamayor~U. 
	\newblock Patching local uniformizations. 
	\newblock {\em Ann.~Sci.~École Norm.~Sup. (4)} 25 (1992), no.~6, 629--677.
	
	\ 
\end{thebibliography}
\end{document}